\documentclass[a4paper,12pt,final]{amsart}
\usepackage{times,a4wide,mathrsfs,amssymb}

\newcommand{\C}{\mathbb{C}}
\newcommand{\ZZ}{\mathbb{Z}}

\newcommand{\QQ}{\mathbb{Q}}
\newcommand{\NN}{\mathbb{N}}

\newcommand{\MM}{\mathcal M}

\newcommand{\wt}{\widetilde}
\newcommand{\ima}{\hbox{Im}}
\newcommand{\rom}{\romannumeral}

\newcommand\undermat[2]{
  \makebox[0pt][l]{$\smash{\underbrace{\phantom{
    \begin{matrix}#2\end{matrix}}}_{\text{$#1$}}}$}#2}

\newtheorem{theorem}{Theorem}[section]

\newtheorem{lemma}[theorem]{Lemma}
\newtheorem{corollary}[theorem]{Corollary}
\newtheorem{proposition}[theorem]{Proposition}
\newtheorem{conjecture}[theorem]{Conjecture}
\newtheorem{remark}[theorem]{Remark}
\newtheorem{definition}[theorem]{Definition}
\newtheorem{convention}{Conventions}

\newtheorem{notation}[theorem]{Notation}

\newtheorem{nonumbering}{Theorem}

\newtheorem{nonumberingt}{Acknowledgements}

\begin{document}
\author[Robert Laterveer]
{Robert Laterveer}

\address{Institut de Recherche Math\'ematique Avanc\'ee,
CNRS -- Universit\'e 
de Strasbourg,\
7 Rue Ren\'e Des\-car\-tes, 67084 Strasbourg CEDEX,
FRANCE.}
\email{robert.laterveer@math.unistra.fr}

\title{Some new examples of smash--nilpotent algebraic cycles}

\begin{abstract} Voevodsky has conjectured that numerical equivalence and smash--equivalence coincide for algebraic cycles on any smooth projective variety. Building on work of Vial and Kahn--Sebastian, we give some new examples of varieties where Voevodsky's conjecture is verified.
\end{abstract}



\maketitle

\section{Introduction}

Let $X$ be a smooth projective variety over $\C$. There exist numerous adequate equivalence relations (in the sense of \cite{Sam}) on the group of algebraic cycles on $X$, ranging from rational equivalence (the finest) to numerical equivalence (the coarsest). Rational equivalence gives rise to the Chow groups $A^j(X):=CH^j(X)_{\QQ}$  (i.e., codimension $j$ cycles with rational coefficients modulo rational equivalence). The other equivalence relations give rise to subgroups $A^j_\sim$, for example there are subgroups
  \[   A^j_{alg}(X)\subset A^j_\otimes(X)\subset A^j_{hom}(X)\subset A^j_{num}(X)\]
  of cycles algebraically resp. smash--nilpotent resp. homologically resp. numerically trivial.
Here, the first inclusion is a theorem of Voevodsky \cite{Voe} and Voisin \cite{V9}, and the last inclusion is the subject of one of the standard conjectures \cite{K}.
More ambitiously, Voevodsky has conjectured that $A^j_\otimes(X)$ and $A^j_{num}(X)$ should coincide \cite{Voe}.

Not a great deal is known about this conjecture of Voevodsky's; most results focus on $1$--cycles. For instance, Voevodsky's conjecture has been proven for $1$--cycles on varieties rationally dominated by products of curves \cite{Seb}, \cite[Proposition 2]{Seb2} (this is further generalized by \cite[Theorem 3.17]{V3}).

In this note (which is inspired by \cite{Seb}, \cite{Seb2} and particularly \cite{V3}), we aim for results for cycles in other dimensions by restricting attention to very special varieties.
 The main result is as follows:

\begin{nonumbering}[$\approx$ theorem \ref{main}\footnote{(The actual statement of theorem \ref{main} is somewhat more general, but this simplified version suffices for many applications.)}] Let  
 $X$ be a smooth projective variety.
 Assume that $X$ is dominated by a product of curves, and that the even cohomology of $X$ verifies
     \[ H^{2i}(X,\QQ)=\wt{N}^{i-1} H^{2i}(X,\QQ)\ .\]
      Then
    \[ A^j_\otimes(X)=A^j_{num}(X)\ \ \hbox{for\ all\ }j\ .\]
      \end{nonumbering}

Here $\wt{N}^\ast$ denotes Vial's niveau filtration, which is a variant of the coniveau filtration (cf. \cite{V4} and section \ref{niv} below). Conjecturally, the condition $H^{2i}(X,\QQ)=\wt{N}^{i-1}$ is equivalent to having $H^{2i}(X,\C)=F^{i-1} H^{2i}(X,\C)$, where $F^\ast$ is the Hodge filtration.

 Examples of varieties to which theorem \ref{main} applies include the following: Fermat hypersurfaces of odd dimension; products of type $X^2_d\times X^n_{d^\prime}$ with $n$ odd  (where $X^n_d$ denotes a Fermat hypersurface of dimension $n$ and degree $d$). 
  Some more examples where theorem \ref{main} applies are given in corollary \ref{ex}.



\vskip0.6cm

\begin{convention} In this note, the word {\sl variety\/} will refer to a reduced irreducible scheme of finite type over $\C$. A {\sl subvariety\/} is a (possibly reducible) reduced subscheme which is equidimensional. 

{\bf All Chow groups will be with rational coefficients}: we denote by $A_jX$ the Chow group of $j$--dimensional cycles on $X$ with $\QQ$--coefficients; for $X$ smooth of dimension $n$ the notations $A_jX$ and $A^{n-j}X$ will be used interchangeably. 

The notations $A^j_{hom}(X)$, $A^j_{num}(X)$, $A^j_{AJ}(X)$, $A^j_{alg}(X)$ and $A^j_\otimes(X)$ will be used to indicate the subgroups of homologically trivial, resp. numerically trivial, resp. Abel--Jacobi trivial resp. algebraically trivial, resp. smash--nilpotent cycles.
The contravariant category of Chow motives (i.e., pure motives with respect to rational equivalence as in \cite{Sc}, \cite{MNP}) will be denoted $\MM_{\rm rat}$.



We will write $H^j(X)$ 
and $H_j(X)$ 
to indicate singular cohomology $H^j(X,\QQ)$,
resp. Borel--Moore homology $H_j(X,\QQ)$.

\end{convention}

\section{Preliminary}

\subsection{Motives of abelian type}

We refer to \cite{Kim}, \cite{An}, \cite{I}, \cite{J4}, \cite{MNP} for the definition of finite--dimensional motive. 
An essential property of varieties with finite--dimensional motive is embodied by the nilpotence theorem:

\begin{theorem}[Kimura \cite{Kim}]\label{nilp} Let $X$ be a smooth projective variety of dimension $n$ with finite--dimensional motive. Let $\Gamma\in A^n(X\times X)_{}$ be a correspondence which is numerically trivial. Then there is $N\in\NN$ such that
     \[ \Gamma^{\circ N}=0\ \ \ \ \hbox{\rm in}\ \  A^n(X\times X)_{}\ \]
     (here, $\circ$ indicates composition of correspondences).
\end{theorem}

 Actually, the nilpotence property (for all powers of $X$) could serve as an alternative definition of finite--dimensional motive, as shown by Jannsen \cite[Corollary 3.9]{J4}.

\begin{conjecture}[Kimura \cite{Kim}]\label{findim} Every smooth projective variety has finite--dimensional motive.
\end{conjecture}

We are still far from knowing this, but at least there are quite a few non--trivial examples:
 
\begin{remark}\label{exfindim} 
The following varieties have finite--dimensional motive: abelian varieties, varieties dominated by products of curves \cite{Kim}, $K3$ surfaces with Picard number $19$ or $20$ \cite{P}, surfaces not of general type with vanishing geometric genus \cite[Theorem 2.11]{GP}, Godeaux surfaces \cite{GP}, certain surfaces of general type with $p_g=0$ \cite{PW}, \cite{V8}, Hilbert schemes of surfaces known to have finite--dimensional motive \cite{CM}, generalized Kummer varieties \cite[Remark 2.9(\rom2)]{Xu},
 3--folds with nef tangent bundle \cite{Iy} (an alternative proof is given in \cite[Example 3.16]{V3}), 4--folds with nef tangent bundle \cite{Iy2}, log--homogeneous varieties in the sense of \cite{Br} (this follows from \cite[Theorem 4.4]{Iy2}), certain 3--folds of general type \cite[Section 8]{V5}, varieties of dimension $\le 3$ rationally dominated by products of curves \cite[Example 3.15]{V3}, varieties $X$ with $A^i_{AJ}(X)_{}=0$ for all $i$ \cite[Theorem 4]{V2}, products of varieties with finite--dimensional motive \cite{Kim}.
\end{remark}

\begin{definition}\label{abtype} Let $X$ be a smooth projective variety of dimension $n$. We say that $X$ has {\em motive of abelian type\/} if $h(X)\in\MM_{\rm rat}$ is in the subcategory generated by the motives of curves.

\end{definition}

\begin{remark} It follows from the fact that curves have finite--dimensional motive that ``motive of abelian type'' implies ``finite--dimensional motive''. The converse is probably not true
(many motives are {\em not\/} of abelian type, cf. \cite[7.6]{D}), yet it is a (somewhat embarassing) fact that all known finite--dimensional motives happen to be of abelian type.


\end{remark}

Various characterizations of motives of abelian type are given in \cite{V3}. One of these is as follows:

\begin{proposition}[Vial \cite{V3}]\label{v1} Let $X$ be a smooth projective variety of dimension $n$. The motive of $X$ is of abelian type if and only if $A^j_{alg}(X)$ is generated, via correspondences, by Chow groups of products of curves, for all $j>\lceil {n\over 2}\rceil$.
\end{proposition}

\begin{proof} This follows from \cite[Theorem 5]{V3}.
\end{proof}

\begin{proposition}[Vial \cite{V3}]\label{v2} Let $X$ be a smooth projective variety of dimension $n$, and assume $X$ has motive of abelian type.  
  Then the motive of $X$ is isomorphic to a direct summand
  \[  h(X)\ \subset\ \bigoplus_j h(A_j)(m_j)\ \ \ \hbox{\rm in}\ \MM_{\rm rat}\ ,\]
  where the $A_j$ are abelian varieties.
  \end{proposition} 

\begin{proof} It suffices to note that for motives of abelian type there is an inclusion
   \[ h(X)\ \subset\ \bigoplus_j h(M_j)(m_j) \ \ \ \hbox{in}\ \MM_{\rm rat}\ ,\]
   where $M_j$ is a product of curves $C_1\times\cdots\times C_{r_j}$ (this follows from \cite[Theorem 4]{V3}, plus \cite[Theorem 3.11]{V3} applied with $l=d:=\dim X$).
   It is well--known this implies proposition \ref{v2}. 
  
  (Indeed, for some $n_i\ge 2g(C_i)$ let $C_i^{[n_i]}$ denote the $n_i$--th symmetric product, and let $J_i$ denote the Jacobian of $C_i$. There exist morphisms
  \[   M:=C_1\times \cdots C_r\ \to\ C_1^{[n_1]}\times\cdots\times C_r^{[n_r]}\ \to\ J_1\times\cdots \times J_r\ .\]
  The first arrow identifies $h(M)$ with a direct summand of $h(C_1^{[n_1]}\times\cdots\times C_r^{[n_r]})$ \cite{KV}. The second arrow is a composition of projective bundles, so the motive $h(C_1^{[n_1]}\times\cdots\times C_r^{[n_r]})$ identifies with a sum of shifted motives of $J_1\times\cdots\times J_r$.)
   \end{proof}

\subsection{Lefschetz standard conjecture}

\begin{notation}
Let $X$ be a smooth projective variety of dimension $n$, and $h\in H^2(X,\QQ)$ the class of an ample line bundle. The hard Lefschetz theorem asserts that the map
  \[  L^{n-i}\colon H^i(X)\to H^{2n-i}(X)\]
  obtained by cupping with $h^{n-i}$ is an isomorphism, for any $i< n$.
  \end{notation}
  
   One of the standard conjectures asserts that the inverse isomorphism is algebraic:

\begin{definition} Given a variety $X$, we say that $B(X)$ holds if for all ample $h$, and all $i<n$ the isomorphism 
  \[  (L^{n-i})^{-1}\colon 
  H^{2n-i}(X)\stackrel{\cong}{\rightarrow} H^i(X)\]
  is induced by a correspondence.
 \end{definition}  
 
 \begin{remark} It is known that $B(X)$ holds for the following varieties: curves, surfaces, abelian varieties \cite{K0}, \cite{K}, threefolds not of general type \cite{Tan}, hyperk\"ahler varieties of 
 $K3^{[n]}$--type \cite{ChM}, $n$--dimensional varieties $X$ which have $A_i(X)_{}$ supported on a subvariety of dimension $i+2$ for all $i\le{n-3\over 2}$ \cite[Theorem 7.1]{V}, $n$--dimensional varieties $X$ which have $H_i(X)=N^{\llcorner {i\over 2}\lrcorner}H_i(X)$ for all $i>n$ \cite[Theorem 4.2]{V2}, products and hyperplane sections of any of these \cite{K0}, \cite{K}.
 \end{remark}
 
 \begin{remark}\label{B}
 Let $X$ be a variety with motive of abelian type. Then $B(X)$ holds. This is because the standard conjecture $B$ can also be formulated for motives. Since $B(A)$ holds for abelian varieties, it also holds for direct summands of a sum of twisted motives of abelian varieties, hence for varieties with motive of abelian type.
 It follows that the standard conjectures $C(X)$ (i.e., algebraicity of the K\"unneth components) and $D(X)$ (i.e., homological and numerical equivalence coincide on $X$ and on 
 $X\times X$) also hold \cite{K0}, \cite{K}. 
 \end{remark}

\subsection{Niveau filtration}\label{niv}

\begin{definition}[Coniveau filtration \cite{BO}]\label{con} Let $X$ be a quasi--projective variety. The {\em coniveau filtration\/} on cohomology and on homology is defined as
  \[\begin{split}   N^c H^i(X,\QQ)&= \sum \ima\bigl( H^i_Y(X,\QQ)\to H^i(X,\QQ)\bigr)\ ;\\
                           N^c H_i(X,\QQ)&=\sum \ima \bigl( H_i(Z,\QQ)\to H_i(X,\QQ)\bigr)\ ,\\
                           \end{split}\]
   where $Y$ runs over codimension $\ge c$ subvarieties of $X$, and $Z$ over dimension $\le i-c$ subvarieties.
 \end{definition}

Vial introduced the following variant of the coniveau filtration:

\begin{definition}[Niveau filtration \cite{V4}] Let $X$ be a smooth projective variety. The {\em niveau filtration} on homology is defined as
  \[ \wt{N}^j H_i(X)=\sum_{\Gamma\in A_{i-j}(Z\times X)_{}} \ima\bigl( H_{i-2j}(Z)\to H_i(X)\bigr)\ ,\]
  where the union runs over all smooth projective varieties $Z$ of dimension $i-2j$, and all correspondences $\Gamma\in A_{i-j}(Z\times X)_{}$.
  The niveau filtration on cohomology is defined as
  \[   \wt{N}^c H^iX:=   \wt{N}^{c-i+n} H_{2n-i}X\ .\]
  
\end{definition}

\begin{remark}\label{is}
The niveau filtration is included in the coniveau filtration:
  \[ \wt{N}^j H^i(X)\subset N^j H^i(X)\ .\] 
  These two filtrations are expected to coincide; indeed, Vial shows this is true if and only if the Lefschetz standard conjecture is true for all varieties \cite[Proposition 1.1]{V4}. 
  
  Using the truth of the Lefschetz standard conjecture in degree $\le 1$, it can be checked \cite[page 415 "Properties"]{V4} that the two filtrations coincide in a certain range:
  \[  \wt{N}^j H^i(X)= N^j H^iX\ \ \ \hbox{for\ all\ }j\ge {i-1\over 2} \ .\]
  \end{remark}

 \begin{lemma}\label{split} Let $X$ be a smooth projective variety of dimension $n$ such that $B(X)$ holds. Suppose 
  \[ H^{2i}(X)= \wt{N}^{ i-1} H^{2i}(X) \]
  for some $i$. Then there exists a smooth projective surface $S_{}$ and correspondences $\Gamma_{2i}\in A^{n+1-i}(X\times S_{})$, $\Psi_{2i}\in A^{i+1}(S_{}\times X)$ such that
   \[ \pi_{2i}=\Psi_{2i}\circ \Gamma_{2i}\ \ \ \hbox{in}\ H^{2n}(X\times X)\ .\]
   \end{lemma}
   
   \begin{proof} This follows readily from the arguments contained in \cite{V4}. Indeed, by assumption there exists a surface $S_{}$ and a correspondence $\Psi_{2i}\in A^{i+1}(S_{}\times X)$
   such that
     \[ H^{2i}(X)=(\Psi_{2i})_\ast H^2(S_{})\ .\]
     This means that the homomorphism of motives
     \[  \Psi_{2i}\colon\ \ (S_{},\pi_2,0)\ \to\  (X,\pi_{2i},0)\ \ \ \hbox{in}\ \MM_{\rm hom} \]
     is surjective (i.e., 
     \[ (\Psi_{2i}\times \Delta_M)_\ast\colon H^\ast(S_{}\times M)\ \to\  (\pi_{2i}\times\Delta_M)_\ast H^\ast(X\times M)\]
      is surjective for all smooth projective varieties $M$). On the other hand, the motives $(S_{},\pi_2,0)$ and $ (X,\pi_{2i},0)$ lie in a subcategory $\MM_{\rm hom}^\circ\subset \MM_{\rm hom}$ which is semi--simple (one can define $\MM_{\rm hom}^\circ$ as the smallest full subcategory containing the motives of all varieties $M$ for which $B(M)$ is known). As such, there is a left--inverse to $\Psi_{2i}$; this gives the correspondence $\Gamma_{2i}$ with the property that $\Psi_{2i}\circ \Gamma_{2i}=\pi_{2i}$.
        \end{proof}

\subsection{Smash--nilpotence}

\begin{definition} Let $X$ be a smooth projective variety. A cycle $a\in A^r(X)$ is called {\em smash--nilpotent\/} 
if there exists $m\in\NN$ such that
  \[ \begin{array}[c]{ccc}  a^m:= &\undermat{(m\hbox{ times})}{a\times\cdots\times a}&=0\ \ \hbox{in}\  A^{mr}(X\times\cdots\times X)_{}\ .
  \end{array}\]
  \vskip0.6cm

We will write $A^r_\otimes(X)\subset A^r(X)$ for the subgroup of smash--nilpotent cycles.
\end{definition}

\begin{conjecture}[Voevodsky \cite{Voe}]\label{voe} Let $X$ be a smooth projective variety. Then
  \[  A^r_{num}(X)\ \subset\ A^r_\otimes(X)\ \ \ \hbox{for\ all\ }r\ .\]
  \end{conjecture}

\begin{remark} It is known \cite[Th\'eor\`eme 3.33]{An} that conjecture \ref{voe} implies (and is strictly stronger than) conjecture \ref{findim}.
\end{remark}

The most general result concerning smash--nilpotence is the following:

\begin{theorem}[Voevodsky \cite{Voe}, Voisin \cite{V9}]\label{VV} Let $X$ be a smooth projective variety. Then
  \[ A^r_{alg}(X)\ \subset\ A^r_\otimes(X)\ \ \ \hbox{for\ all\ }r\ .\]
  \end{theorem}
  
  In particular, it follows from theorem \ref{VV} that conjecture \ref{voe} is true for $r=1$ and for $r=\dim X$. Another useful result is the following (this is \cite[Proposition 1]{KS}, which builds on results of Kimura's \cite{Kim}):

\begin{theorem}[Kahn--Sebastian \cite{KS}]\label{skew} Let $A$ be an abelian variety. Assume $a\in A^r(A)$ is {\em skew\/}, i.e. $(-1)^\ast(a)=-a$ in $A^r(A)$. Then $a\in A^r_\otimes(A)$.
\end{theorem}

\section{Main result}

This section contains the proof of our main result (stated in somewhat more general form than in the introduction):

\begin{theorem}\label{main} Let $X$ be a smooth projective variety of dimension $n$. Assume


\noindent
{(\rom1)} $X$ has motive of abelian type;


\noindent
{(\rom2)} $H^{2i}(X)=\wt{N}^{i-1} H^{2i}(X)$ for all $i\le n/2$.

Then Voevodsky's conjecture is true for $X$, i.e.
  \[ A^r_\otimes(X)=A^r_{num}(X)\ \ \hbox{for\ all\ }r\ .\]
  \end{theorem}

\begin{proof} 
Let us denote
  \[ Z^r(X):= { A^r_{num}(X)\over A^r_\otimes(X) }\ .\]

By assumption (\rom1), the K\"unneth components $\pi_i$ of $X$ are algebraic (remark \ref{B}). 
By assumption (\rom2) and lemma \ref{split}, any ``even'' K\"unneth component $\pi_{2i}$ with $i\le n/2$
factors over a surface, i.e. there exists a surface $S_{2i}$ and correspondences $\Gamma_{2i}\in A^{n+1-i}(X\times S_{2i})$, $\Psi_{2i}\in A^{i+1}(S_{2i}\times X)$ such that
   \[ \pi_{2i}=\Psi_{2i}\circ \Gamma_{2i}\ \ \ \hbox{in}\ H^{2n}(X\times X)\ .\]
  
  We now lift the $\pi_i$ to the level of rational equivalence in the following way: for the even components we choose 
    \[    \Pi_{2i}:=\begin{cases} \ \ \  \Psi_{2i}\circ \Gamma_{2i}  \ \ \ \ \ \ \ \ \ \hbox{in}\ A^{n}(X\times X) &\hbox{if\ }i\le n/2\ ;\\
                                                   {}^t  (\Psi_{2n-2i}\circ \Gamma_{2n-2i})  \ \ \ \hbox{in}\ A^{n}(X\times X)&\hbox{if}\ i> n/2\ , \\
                                                   \end{cases}\]
    Here $\Psi_{2i}, \Gamma_{2i}$ are correspondences to and from a surface $S_{2i}$ as above, and ${}^t()$ denotes the transpose of a correspondence.
   For the odd K\"unneth components $\pi_{2i+1}$, we take arbitrary lifts $\Pi_{2i+1}\in A^n(X\times X)$ of the $\pi_{2i+1}$, subject only to the condition that
   \[  \Delta_X={\displaystyle \sum_{i=0}^{2n}} \Pi_i\ \ \hbox{in}\ A^n(X\times X) \]
   (i.e., we define the last $\Pi_{2i+1}$ as a difference of cycle classes). Note that our $\Pi_i\in A^n(X\times X)$ need {\em not\/} be idempotents.
  
  We now remark that
    \[ (\Pi_{2i})_\ast\colon\ \ Z^r(X)\ \to\ Z^r(X) \]
    factors over $Z^\ast(S_{2i})$, which is $0$ since $S_{2i}$ is a surface, and so
    \[  (\Pi_{2i})_\ast=0\colon\ \ Z^r(X)\ \to\ Z^r(X)\ \ \ \hbox{for\ all\ }i\ \hbox{and\ all\ }r\ .\]
   It follows that
    \[ (\Delta_X)_\ast=({\displaystyle \sum_{i\ {\rm odd}}} \Pi_i)_\ast\colon\ \ Z^r(X)\ \to\ Z^r(X)\ .\]
    For later use, let us note that this last equality also implies 
    \begin{equation}\label{power}
    ({\displaystyle \sum_{i\ {\rm odd}}} \Pi_i)_\ast=\Bigl(({\displaystyle \sum_{i\ {\rm odd}}} \Pi_i)^{\circ m}\Bigr){}_\ast=\hbox{id}\colon\ \ Z^r(X)\ \to\ Z^r(X)\ ,\ \ \hbox{for\ all\ }m\in \NN\ .
    \end{equation}

Assumption (\rom1) implies the motive of $X$ identifies with a direct summand
  \[ h(X)\ \subset\ \bigoplus_{j=1}^s h(A_j)(m_j)\ \ \ \hbox{in}\ \ \MM_{\rm rat}\ ,\]
  where the $A_j$ are abelian varieties (proposition \ref{v2}). This formally implies that there exist correspondences
    \[  \begin{split} \Gamma_1&={\displaystyle\sum_j} \Gamma_1^j\ \ \in \bigoplus_j A^\ast(X\times A_j)\ ,\\
                            \Gamma_2&={\displaystyle\sum_j} \Gamma_2^j\ \ \in \bigoplus_j A^\ast(A_j\times X)\\
                            \end{split}\]
                     such that
            \[ \Gamma_2\circ \Gamma_1={\displaystyle\sum_j} \Gamma_2^j\circ \Gamma_1^j=\Delta_X\ \ \hbox{in}\ A^n(X\times X)\ .\]
       In particular, for any $i$ we also have that the composition
       \[  H^{2i+1}(X)\ \xrightarrow{(\Gamma_1)_\ast}\ \bigoplus_j H^{2i+1+2c_j}(A_j)\ \xrightarrow{(\Gamma_2)_\ast}\ H^{2i+1}(X)\]
       is equal to the identity (here $c_j$ is some integer, dependent on $n$ and $\dim A_j$ and $m_j$).     
       But this composition is the same as
       \[    \begin{split}  H^{2i+1}(X)\ \xrightarrow{(\Gamma_1)_\ast}\ \bigoplus_j H^{2i+1+2c_j}(A_j)&\ 
       \xrightarrow{\bigl((\pi_{2i+1+2c_1}^{A_1})_\ast,\ldots, (\pi_{2i+1+2c_s}^{A_s})_\ast\bigr)}\ \\
           &\bigoplus_j H^{2i+1+2c_j}(A_j)\ \xrightarrow{(\Gamma_2)_\ast}\ H^{2i+1}(X)\ ,\\
           \end{split}\]
        where the $\pi^{A_j}_i$ denote the Chow--K\"unneth decomposition of \cite{Mur} for the abelian variety $A_j$.  
           
  That is, we have a homological equivalence
    \[  \Pi_{2i+1}=\Gamma_2\circ \Gamma_1\circ \Pi_{2i+1}=
    {\displaystyle\sum_j}\ \Gamma_2^j\circ \pi^A_{2i+1+2c_j}\circ \Gamma_1^j\circ \Pi_{2i+1}\ \ \hbox{in}\ H^{2n}(X\times X)\ .\]
   Taking the sum over all odd K\"unneth components, we find that
     \[ {\displaystyle  \sum_{i\ {\rm odd}}} \Pi_{i}     - {\displaystyle  \sum_{i\ {\rm odd}}}  {\displaystyle \sum_j} \Gamma_2^j\circ \pi^{A_j}_{2i+1+2c_j}\circ \Gamma_1^j\circ \Pi_{2i+1}\ \ \in\  A^n_{}(X\times X)\ \]
     is homologically trivial. But then (since $X$ has finite--dimensional motive), it follows from theorem \ref{nilp} this cycle is nilpotent: there exists $N\in\NN$ such that
     \[ \Bigl({\displaystyle  \sum_{i\ {\rm odd}}} \Pi_{i}     -{\displaystyle \sum_{i\ {\rm odd}}}{\displaystyle \sum_j} \Gamma_2^j\circ \pi^{A_j}_{2i+1+2c_j}\circ \Gamma_1^j\circ \Pi_{2i+1}\Bigr)^{\circ N}=0\ \ \in\  A^n_{}(X\times X)\ .\]
          Developing this expression, we obtain
      \[     \bigl( {\displaystyle \sum_{i\ {\rm odd}}}\Pi_{i}\bigr)^{\circ N}=  Q_1+\cdots + Q_{N^\prime}
       \ \ \hbox{in}\ A^n(X\times X)\ ,\]
       where each $Q_s$ is a composition of correspondences in which some $\pi^{A_j}_{2i+1+2c_j}$ occurs at least once, i.e.
       \[ Q_s= (\hbox{something})\circ \pi^{A_j}_{\ell}   \circ (\hbox{something})\ \ \ \hbox{in}\ A^n(X\times X)\ \ \ ,\hbox{with\ $\ell$\ odd}.\] 
                This equality implies in particular that both sides act in the same way on $Z^r(X)$ for any $r$, i.e.
      \[ \Bigl( \bigl( {\displaystyle \sum_{i\ {\rm odd}}}\Pi_{i}\bigr)^{\circ N}\Bigr){}_\ast= ( {\displaystyle \sum_{s}} Q_s)_\ast  = (\hbox{something})_\ast
      (\pi^{A_j}_{\ell})_\ast(\hbox{something}){}_\ast\colon \ \  Z^r(X)\ \to\ Z^r(X)\ .\]      
      
      The right--hand--side of this equality is $0$, since 
         \[ (\pi^{A_j}_{\ell})_\ast \Bigl( {A^\ast_{num}(A_j)\over A^\ast_\otimes(A_j)}\Bigr)\]
    for $\ell$ odd (this follows from theorem \ref{skew}, combined with the fact that the $\pi^{A_j}_\ell$ project to odd gradeds of the Beauville filtration on $A^\ast(A_j)$ \cite{DM}, \cite{Sc}). As we have seen in equality (\ref{power}), the left--hand--side is the identity. We conclude that 
      \[ Z^r(X)=0\ .\]      
   \end{proof}

\section{Examples}

In this section, we aim to give some content to theorem \ref{main}, by providing examples of varieties satisfying the assumptions. For convenience, we will write $X^n_d$ for the Fermat hypersurface of dimension $n$ and degree $d$.

\begin{corollary}\label{ex} Let $X$ be one of the following:

\noindent
(1) a Fermat hypersurface $X^n_d$ with $n$ odd;

\noindent
(2) a product $Y_1\times\cdots\times Y_s\times X^n_d$, where the $Y_i$ are varieties with $A^\ast_{hom}(Y_i)=0$ (examples of such varieties can be found in \cite{BKL}, \cite{PW}, \cite{V8}), and $n$ is odd;

\noindent
(3) a product $Y_1\times\cdots\times Y_s\times Y$, where the $Y_i$ are as in (2), and $Y$ is a Calabi--Yau threefold with motive of abelian type (examples of such $Y$ are given in \cite[Section 2]{des} and in \cite{exCY});

\noindent
(4) a product $S\times X^n_{d}$ where $n$ is odd, and $S$ is a regular surface with motive of abelian type (e.g., $S$ can be $X^2_{d^\prime}$, or a double plane branched along $6$ lines in general position \cite{Par}, or a $K3$ surface with Picard number $\ge 19$, or any of the surfaces in \cite{GaP}, \cite{Bonf});

\noindent
(5) a product $Y\times C$, where $C$ is a curve and $Y=X^4_7/\mu_7$ is the fourfold studied in \cite[Proposition 2.17]{V4};

\noindent
(6) a product $Y\times S$, where $S$ is a surface with $A^2_{AJ}(S)=0$, and $Y=X^4_7/\mu_7$ is the fourfold of \cite[Proposition 2.17]{V4};

\noindent
{(7)} a product $S\times Y$, where $S$ is a regular surface with motive of abelian type, and $Y$ is a Calabi--Yau threefold with motive of abelian type; 


\noindent
{(8)} the Calabi--Yau $5$--fold obtained from a product of $5$ elliptic curves as in \cite[Corollary 2.3]{CH}.

Then 
  \[ A^r_\otimes(X)=A^r_{num}(X)\ \ \hbox{for\ all\ }r\ .\]
\end{corollary}

\begin{proof} Clearly, all these examples have motive of abelian type: for case (1), this follows from Shioda's inductive structure \cite{Shi}; for case (2) this follows from \cite[Theorem 5]{V2} (or, independently, from \cite{Kim2}); case (5) follows from \cite[Proposition 2.17]{V4}; the surfaces in case (6) follow from \cite[Theorem 4]{V2}. It remains to check hypothesis (\rom2) of theorem \ref{main} is verified. For cases (1), (2), (3) this is clear since in these cases the even--degree cohomology is algebraic, and so
  \[ H^{2i}(X)=N^i H^{2i}(X)=\wt{N}^i H^{2i}(X)\subset \wt{N}^{i-1} H^{2i}(X)\ .\]
  
  In case (4), we have
   \[  H^{2i}(X)=\bigoplus_{k+\ell=2i} H^k(S)\otimes H^{\ell}(X^n_{d^\prime})\ .\]
 Any direct summand with $k\not=2$ consists of algebraic classes. For $k=2$, we have
   \[ H^2(X^2_d)\otimes  H^{2i-2}(X^n_{d^\prime})\ \subset\ H^2(X^2_d)\otimes \wt{N}^{i-1} H^{2i-2}(X^n_{d^\prime})\ \subset\   \wt{N}^{i-1} H^{2i}(X^2_d\times X^{n}_{d^\prime})\ .\]  
 
 In case (5), we have $H^4(Y)=\wt{N}^1 H^4(Y)$ \cite[Proposition 2.17]{V4}. It follows that
   \[ H^4(Y\times C)= H^4(Y)\otimes H^0(C)\oplus H^2(Y)\otimes H^2(C)\ \subset\ \wt{N}^1 H^4(Y\times C)\ .\]
   
   In case (6), we have 
     \[ \begin{split} H^4(X)&= H^4(Y)\otimes H^0(S)\oplus H^2(Y)\otimes H^2(S)\oplus H^0(Y)\otimes H^4(S)\\ 
                               &\subset\ \wt{N}^1 H^4(Y)\otimes H^0(S)\oplus \wt{N}^1 H^2(Y)\otimes \wt{N}^1 H^2(S)\oplus H^0(Y)\otimes \wt{N}^2 H^4(S)\\
                               &\subset \wt{N}^1 H^4(X)\ ,\\
                               \end{split}\]
                   and likewise
 \[ \begin{split} H^6(X)&= H^6(Y)\otimes H^0(S)\oplus H^4(Y)\otimes H^2(S)\oplus H^2(Y)\otimes H^4(S)  \\ 
                               &\subset\ \wt{N}^2 H^6(Y)\otimes H^0(S)\oplus \wt{N}^1 H^4(Y)\otimes \wt{N}^1 H^2(S)\oplus \wt{N}^1 H^2(Y)\otimes \wt{N}^2 H^4(S)\\
                               &\subset \wt{N}^2 H^6(X)\ .\\
                               \end{split}\]

 Cases (7) is similar to case (4).
 
As to case (8): let $E_1,\ldots, E_5$ be elliptic curves, and let $X$ be a Calabi--Yau $5$--fold obtained as a smooth model of the quotient
  \[   (E_1\times\cdots\times E_5)/\ZZ_2^4\]
  as in \cite[Corollary 2.3]{CH}. It is readily checked (using the argument of \cite[Lemma 2.4]{CH}) that
  \[  H^4(E_1\times\cdots\times E_5)^{\ZZ_2^4}\subset N^2 H^4(E_1\times\cdots\times E_5)\ . \]
Next, the inductive construction of \cite[Proposition 2.1]{CH}  shows $X$ is of the form $Z/\ZZ_2^4$, where $Z$ is obtained from $E_1\times\cdots\times E_5$ by blowing up some rational subvarieties. Since rational varieties of dimension $\le 3$ verify the Lefschetz standard conjecture, this implies
  \[  H^i(Z)\ \subset\ \ima \bigl( H^i(E_1\times\cdots\times E_5)\ \to\ H^i(Z)\bigr) \ \cup \ \wt{N}^1 H^i(Z)\ \ \ \hbox{for\ all\ }i\ .\]
  In particular, it follows that
  \[ H^4(X)=H^4(Z)^{\ZZ_2^4}\ \subset\   \ \wt{N}^1 H^4(X)\ .\]
     \end{proof}

\vskip1cm

\begin{nonumberingt} The ideas developed in this note grew into being during the Strasbourg 2014---2015 groupe de travail based on the monograph \cite{Vo}. Thanks to all the participants of this groupe de travail for the pleasant and stimulating atmosphere. 
Thanks to Charles Vial and an anonymous referee for helpful comments.
Many thanks to Yasuyo, Kai and Len for not disturbing when the door to my office is closed.
\end{nonumberingt}

\end{document}